\documentclass[12pt]{amsart}
\usepackage[DIV=12,oneside]{typearea}
\usepackage{hyperref}
\hypersetup{%
   colorlinks,%
   citecolor=blue,%
   urlcolor=blue%
}
\usepackage{xcolor}
\usepackage[kerning,spacing]{microtype}
\microtypecontext{spacing=nonfrench}
\usepackage{amssymb}
\usepackage{booktabs}
\usepackage{bbm}
\usepackage{graphicx}
\usepackage{tabularx}
\usepackage{url}
\usepackage{xspace}

\makeatletter
 \def\@textbottom{\vskip \z@ \@plus 1pt}
 \let\@texttop\relax
\makeatother

\usepackage{tikz}
\usepackage{enumitem}

\usepackage{eqparbox}
\usepackage[numbers]{natbib}

\numberwithin{equation}{section}

\newtheorem{theorem}{Theorem}
\newtheorem{proposition}[theorem]{Proposition}
\newtheorem{lemma}[theorem]{Lemma}
\newtheorem{cor}[theorem]{Corollary}

\theoremstyle{definition}

\newtheorem{remark}[theorem]{Remark}

\newtheorem{definition}[theorem]{Definition}

\numberwithin{theorem}{section}

\newcommand{\polymake}{\texttt{poly\-make}\xspace}

\newcommand{\pro}[2]{\langle #1, #2 \rangle}
\newcommand{\N}{\mathbb{N}}

\newcommand{\C}{\mathbb{C}}
\newcommand{\Z}{\mathbb{Z}}
\newcommand{\Q}{\mathbb{Q}}
\newcommand{\R}{\mathbb{R}}

\renewcommand{\P}{\mathbb{P}}

\newcommand{\onev}{\mathbf{1}}

\newcommand{\V}{\mathcal{V}}

\newcommand{\Est}{E_{\rm st}}

\def\MR{M_\R}
\def\NR{N_\R}

\def\hstpq{h_{\rm st}^{p,q}}

\def\lolra{{\Longleftrightarrow}}

\DeclareMathOperator{\conv}{conv}
\DeclareMathOperator{\cone}{cone}
\DeclareMathOperator{\aff}{aff}

\title[On smooth Gorenstein polytopes]{On smooth Gorenstein polytopes}

\author{Benjamin Lorenz}
\address{Goethe-Universit\"at, FB 12 - Institut f\"ur Mathematik,
Robert-Mayer-Stra\ss e 10, D-60325 Frankfurt am Main}
\email{\href{mailto:blorenz@math.uni-frankfurt.de}{blorenz@math.uni-frankfurt.de}}

\author{Benjamin Nill}
\address{Case Western Reserve University, Math Department, Yost Hall 242,
10900 Euclid Avenue, Cleveland, Ohio 44106-7058}
\email{\href{mailto:benjamin.nill@case.edu}{benjamin.nill@case.edu}}

\begin{document}

\begin{abstract}
A Gorenstein polytope of index $r$ is a lattice polytope whose $r$th dilate is a reflexive polytope. These objects are of interest in combinatorial commutative algebra and enumerative combinatorics, and play a crucial role in Batyrev's and Borisov's computation of Hodge numbers of mirror-symmetric generic Calabi-Yau complete intersections. In this paper, we report on what is known about smooth Gorenstein polytopes, i.e., Gorenstein polytopes whose normal fan is unimodular. We classify $d$-dimensional smooth Gorenstein polytopes with index larger than $(d+3)/3$. Moreover, we use a modification of {\O}bro's algorithm to achieve classification results for smooth Gorenstein polytopes in low dimensions. The first application of these results is a database of all toric Fano $d$-folds whose anticanonical divisor is divisible by an integer $r$ satisfying $r \ge d-7$. As a second application we verify that there are only finitely many families of Calabi-Yau complete intersections of fixed dimension that are associated to a smooth Gorenstein polytope via the Batyrev-Borisov construction. 
\end{abstract}

\maketitle

\section{Introduction from a combinatorial viewpoint}

\subsection{The combinatorial setting}
\label{comb-intro}
Let us start by introducing the notions of reflexive and Gorenstein polytopes. We refer the reader to \cite{BN08} and \cite{Nil05} for more details and the algebro-geometric background. Let us fix a dual pair of lattices $M$ and $N$. We define $\MR = M \otimes_\Z \R$ and its dual vectorspace $\NR$ analogously. A {\em lattice polytope} $P \subset \MR$ is the convex hull of finitely many lattice points (elements in the lattice $M$). If such a lattice polytope contains the origin in its interior, then the {\em dual polytope} $P^*$ is defined as
\[P^* = \{ y\in \NR \,|\, \langle y,x \rangle \geq -1 \text{ for all } x\in P\}.\]
We remark that $P^*$ does not have to be a lattice polytope anymore. A {\em reflexive polytope} $P \subseteq \MR$ is a lattice polytope containing the origin in its interior such that $P^*$ is also a lattice polytope. A lattice polytope $P$ is {\em Gorenstein of index $r$}, if $r P$ is a reflexive polytope up to translation by a lattice point. This index is uniquely determined. In other words, Gorenstein polytopes of index $r$ are in one-to-one correspondence to reflexive polytopes 
that are `divisible' by $r$ (i.e., if $v$ is a vertex of such a reflexive polytope $\tilde{P}$, then $(\tilde{P}-v)/r$ is a lattice polytope). Two lattice polytopes are considered {\em isomorphic} (or {\em unimodularly equivalent}), if there is a lattice automorphism mapping their vertex sets onto each other. 
In each dimension there exist only finitely many reflexive polytopes up to isomorphisms. They are known up to dimension four by massive 
computer calculations by Kreuzer and Skarke \cite{classreflexive3, classreflexive4}. Recently, Skarke described a procedure how to possibly extend their algorithm to Gorenstein polytopes \cite{skarke}.

\subsection{Smooth Gorenstein polytopes of large index}
A $d$-polytope is called {\em simple}, if each vertex is contained in precisely $d$ edges. A $d$-dimensional lattice polytope $P$ is called {\em smooth}, if $P$ is simple and at each vertex $v$ the primitive edge directions $v_1-v, \ldots, v_d-v$ form a lattice basis. Alternatively, the associated normal fan is unimodular, respectively, its associated toric variety is nonsingular. We refer to \cite{toricbook} for standard results in toric geometry.

\medskip

{\O}bro described in \cite{oebro} an algorithm which has been used to classify all smooth reflexive polytopes for $d\le 8$. With an improved implementation, this was extended to $d=9$ by Andreas Paffenholz and the first author, see~\cite{sfpdim9}. In this paper we apply a slightly modified version of his algorithm to compute high-dimensional smooth reflexive polytopes that are highly divisible. The algorithmic outcome of these computations is summarized in Table~\ref{fano-table}. The database of these polytopes can be found online \cite{sgpdb}.

\begin{table}[tbp]
{\small
\begin{center}
\begin{tabular}{r|rrrrrrrrrrrrrr}
\begin{tikzpicture}[scale=0.5,trim right=0.77cm,baseline=2pt]
\draw[black,thin] (1.85,-0.1) -- (0.3,1);
\draw (0.3,0.4) node {\footnotesize $d$};
\draw (1.4,0.8) node {\footnotesize $r$};
\end{tikzpicture}&13&12&11&10&9&8&7&6&5&4&3&2&1&\\
\hline    
20&0&0&1&2&\textbf{5}&11&&&&&&&\\
19&0&0&0&2&3&7&&&&&&&&\\
18&0&0&0&1&2&\textbf{5}&&&&&&&&\\
17&0&0&0&0&2&3&7&&&&&&&\\
16&0&0&0&0&1&2&\textbf{5}&&&&&&&\\
15&0&0&0&0&0&2&3&&&&&&&\\
14&0&0&0&0&0&1&2&\textbf{5}&&&&&&\\
13&0&0&0&0&0&0&2&3&&&&&&\\
12&1&0&0&0&0&0&1&2&\textbf{6}&&&&&\\
11&&1&0&0&0&0&0&2&3&14&&&&\\
10&&&1&0&0&0&0&1&2&\textbf{6}&64&&&\\
9&&&&1&0&0&0&0&2&4&23&896&8229721&\\
8&&&&&1&0&0&0&1&2&\textbf{13}&258&749892&\\
7&&&&&&1&0&0&0&2&4&85&72256&\\
6&&&&&&&1&0&0&1&3&\textbf{28}&7622&\\
5&&&&&&&&1&0&0&2&12&866&\\
4&&&&&&&&&1&0&1&4&\textbf{124}&$\mathbf{n=3}$\\
3&&&&&&&&&&1&0&3&18&\\
2&&&&&&&&&&&1&1&5&\\
1&&&&&&&&&&&&1&1&\\
0&&&&&&&&&&&&&1&
\end{tabular}\medskip
\caption{Number of isomorphism classes of smooth Gorenstein polytopes of dimension $d$ and index $r$. \label{fano-table}}
\end{center}}
\end{table}

\smallskip

The reader might have noticed certain regularities in the table, if the index $r$ is large. For instance, there is only one smooth Gorenstein polytope (the unimodular simplex, see Section~\ref{sec-cay}) that satisfies $r > \lceil\frac{d+1}{2}\rceil$. These observations can be easily explained and follow from some well-known results in toric geometry. We refer to Theorem~\ref{main-theo} in Section~\ref{proofsec} for a description of all smooth Gorenstein polytopes of index $r > \frac{d+3}{3}$.  

\subsection{Motivation and background from a combinatorial viewpoint}

In this paper we report on what is known for smooth Gorenstein polytopes of either `small' dimension or `large' index. Our goal is to provide the community with a database of interesting examples. Let us describe three areas of relevance, where Gorenstein polytopes turn up. So far, all smooth Gorenstein polytopes we checked confirmed the below mentioned conjectures. 

\begin{itemize}
\item {\em Combinatorial commutative algebra and enumerative combinatorics:} 
Al\-ge\-bra\-i\-cal\-ly, being Gorenstein is equivalent to the fact that the semigroup algebra associated to the cone over a lattice polytope is a Gorenstein algebra, e.g., \cite{MS05}. Combinatorially, this translates into the symmetry of the so-called Ehrhart $h^*$-polynomial \cite{hibi}. The best-known example of a Gorenstein polytope is the Birkhoff polytope, the polytope of all doubly stochastic $n \times n$-matrices (it has dimension $n^2-2n+1$ and index $n$). Gorenstein polytopes and the notion of special simplices have gained increased interest in \cite{BrRo07, OH04, RW05, CHT06, Pay08, BN08, BBKSZ12}, pushed by the proof of Stanley's conjecture on the unimodality of the coefficients of the Ehrhart $h^*$-polynomial of the Birkhoff polytope \cite{Ath05}. It remains an open question, whether unimodality holds for any normal Gorenstein polytope \cite{MP05}, cf.~Definition~\ref{def-normal}.

\smallskip

\item {\em Ehrhart theory and Cayley polytopes:} The index of a Gorenstein polytope is a special case of the codegree of a lattice polytope, which is defined as the minimal positive integer by which you have to multiply a lattice polytope so that it contains an interior lattice point. Lattice polytopes of large codegree have been intensively studied over the last years, since the notion of the codegree is closely related to Ehrhart polynomials. Lattice polytopes with codegree $\ge d$ \cite{BN07} and Gorenstein polytopes of index $\ge d-1$ \cite{classdelpezzo} are completely known. The crucial notion appearing in these studies is that of a {\em Cayley polytope}, we refer to Section~\ref{sec-cay} for the precise definition. A Cayley polytope should be thought of as a lattice polytope which decomposes into smaller-dimensional lattice polytopes. It is conjectured (but still unproven) that a lattice polytope of codegree $\ge \frac{d+3}{2}$ is a Cayley polytope. For Gorenstein polytopes it is known to hold even for index $\ge \frac{d+2}{2}$ \cite{HNP09}. Because of their natural relation to toric manifolds also smooth lattice polytopes have been investigated more closely, and it is known that, if their codegree satisfies $\ge \frac{d+3}{2}$, then they are Cayley polytopes of a certain type \cite{DiR06,DDiRP09,DN10,araujo-new}. However, still basic questions on smooth polytopes are open \cite{gubeladze} such as whether any smooth lattice polytope is normal. 

\smallskip

\item {\em Combinatorial mirror symmetry:} Gorenstein polytopes were originally introduced in \cite{BB97} as an combinatorial framework incorporating the mirror symmetry phenomenon originally observed for Calabi-Yau manfolds. Gorenstein polytopes have a beautiful duality, and it is possible to define for any Gorenstein polytope a {\em stringy E-polynomial} \cite{stringcohomology,BN08,schepers} that has the same symmetry properties as the $E$-polynomial encoding the Hodge numbers of mirror-symmetric Calabi-Yau manifolds. Roughly speaking, these polynomials are sums over pairs of dual faces that sum up products of toric $g$-polynomials and Ehrhart $h^*$-poly\-no\-mi\-als. In other words, they encode combinatorial as well as algebraic information of the Gorenstein polytope and its dual simultaneously. Because of their complicated nature, it seems to be very hard to show statements that are obvious from an algebro-geometric viewpoint (cf.~Lemma~\ref{reduction}) by using only their explicit combinatorial definition. Open questions about these polynomial invariants can be found in \cite{BN08,schepers,skarke}.
\end{itemize}

\section{Introduction from the viewpoint of toric geometry}

\subsection{Toric Fano manifolds of large index}

A {\em Gorenstein Fano variety} $X$ is a $d$-dimensional projective complex variety such that its anticanonical divisor $-K_X$ is an ample Cartier divisor. The {\em index} $i_X$ of $X$ is defined as the largest positive integer $r$ such that there exists some Cartier divisor $D$ with $-K_X = r D$. Of special importance in algebraic geometry are {\em Fano manifolds}, i.e., nonsingular (Gorenstein) Fano varieties. They are completely known up to dimension $3$. In higher dimension, much work has been done to classify all Fano manifolds with large index. If $X$ is smooth, then $i_X \le d+1$ with equality only for $\P^d$. By now, Fano manifolds with $i_X \ge d-2$ or $i_X \ge \frac{d+1}{2}$ are completely known. We refer for these results to \cite{fanolargeindex} and Section~2 in \cite{classhighindex} with the references therein.

\medskip

In the toric situation more can be shown. We refer to \cite[Section 8.3]{toricbook}, and \cite{fanopolytopes} for surveys on toric Fano varieties. Isomorphism classes of toric Fano manifolds $X$ correspond bijectively to isomorphism classes of smooth reflexive polytopes $P$. Here, the index $i_X$ corresponds to the maximal $r$ for which $P$ is `divisible' by $r$ (as defined in Section~\ref{comb-intro}). In particular, any such $X$ is given by a smooth Gorenstein polytope of index $i_X$.  However, note that several smooth Gorenstein polytopes may define the same toric Fano variety (e.g., the reflexive polytope associated to $\P^3$ is divisible by $1$, $2$ and $4$). In other words, the reader should not confuse the index $i_X$ of $X$ with the index $r$ of a Gorenstein polytope. 

\smallskip

From the classification of smooth Gorenstein polytopes in Section~\ref{sec-cay} we can deduce the following observation. 

\begin{proposition}
\label{main-fano}
Let $X$ be a toric Fano $d$-fold. Then $i_X > \frac{d+3}{3}$ if and only if $X \cong \P^d$, or $X \cong \P^{\frac{d}{2}} \times \P^{\frac{d}{2}}$ 
(for $d \ge 4$ even), or 
\[X \cong \P_{\P^{d+1-r}}(O(a_1) \oplus \cdots \oplus O(a_t) \oplus O^{r-t}),\]
where $a_1, \ldots, a_t$ is some integer partition of $d+2-2r$ for $\frac{d+3}{3} < r \le \frac{d+1}{2}$. In the last case, $i_X = r$.
\end{proposition}

Here, an {\em integer partition} of $N$ is a multiset of nonnegative integers summing up to $N$. Let us remark that we do not claim that the previous result is especially original, in fact, it seems to be folklore to experts in toric geometry, even if we couldn't find a reference in the literature. 

\medskip

By now, toric Fano manifolds are known up to dimension $9$ by {\O}bro's algorithm \cite{oebro,sfpdim9} extending previous classifications \cite{wat-wat, bat-3d, class4folds, sato, kn09}. 
As an application of our algorithmic results (Table~\ref{fano-table}) combined with the previous proposition we can determine all non-isomorphic toric Fano manifolds with $i_X \ge d-7$.

\begin{cor}
\label{cor-fano}
Let $X$ be a toric Fano $d$-fold with index $i_X$.

\begin{enumerate}
\item Let $i_X = d$ (cf.~\cite{characterizations}). Then $X \cong \P^1 \times \P^1$.
\item {\em (Toric del Pezzo manifolds, cf.~\cite{classdeltagenus, smalldeltagenera, classdelpezzo})} Let $i_X=d-1$. Then
$$\begin{array}{c|c|c|c}
d & 2 & 3 & 4\\
\hline X & \#=3 & \P^1 \times \P^1 \times \P^1, \P_{\P^2}(O(1) \oplus O) & \P^2 \times \P^2
\end{array}
$$
\item {\em (Toric Mukai manifolds, cf.~\cite{biregularclass, fanovar})} Let $i_X=d-2$. Then
$$\begin{array}{c|c|c|c|c}
d & 3 & 4 & 5 & 6\\
\hline X & \#=15 & \P^1 \times \P^1 \times \P^1 \times \P^1, \P_{\P^3}(O(2) \oplus O),  & \P_{\P^3}(O(1) \oplus O^2) & \P^3 \times \P^3\\
 & & \P^1\times\P^3, \P^1\times\P_{\P^2}(O(1) \oplus O) & &
\end{array}
$$
\item Let $i_X=d-3$. Then 
$$\begin{array}{c|c|c|c|c|c}
d & 4 & 5 & 6 & 7 & 8\\
\hline X & \#=118 & \#=11 & \#=3 & \P_{\P^4}(O(1) \oplus O^3) & \P^4 \times \P^4
\end{array}
$$
\item Let $i_X=d-4$. Then 
$$\begin{array}{c|c|c|c|c|c|c}
d & 5 & 6 & 7 & 8 & 9 & 10\\
\hline X & \#=853 & \#=27 & \#=4 & \#=2 & \P_{\P^5}(O(1) \oplus O^4) & \P^5 \times \P^5
\end{array}
$$
\item Let $i_X=d-5$. Then 
$$\begin{array}{c|c|c|c|c|c|c|c}
d & 6 & 7 & 8 & 9 & 10 & 11 & 12\\
\hline 
\text{no.~of }X & 7590 & 83 & 12 & 4 & 2 & 1 & 1 
\end{array}
$$
\item Let $i_X=d-6$. Then 
$$\begin{array}{c|c|c|c|c|c|c|c|c}
d & 7 & 8 & 9 & 10 & 11 & 12 & 13 & 14\\
\hline 
\text{no.~of }X & 72167 & 256 & 23 & 6 & 3 & 2 & 1 & 1
\end{array}
$$
\item Let $i_X=d-7$. Then 
$$\begin{array}{c|c|c|c|c|c|c|c|c|c}
d & 8 & 9 & 10 & 11 & 12 & 13 & 14 & 15 & 16\\
\hline 
\text{no.~of }X & 749620 & 891 & 63 & 13 & 6 & 3 & 2 & 1 & 1
\end{array}
$$
\end{enumerate}
\end{cor}

The database of the associated smooth Gorenstein polytopes can be found online~\cite{sgpdb}.

\subsection{Complete intersection Calabi-Yau manifolds}

Over the last three decades mirror symmetry has spurred the interest in finding all possible Hodge numbers of Calabi-Yau $n$-folds, in particular, for $n=3$. The `Hodge diamond' of a Calabi-Yau threefold is completely described by the pair of Hodge numbers $(h^{1,1},h^{1,2})$. So far ten thousands of these pairs have been found \cite{jurke,k3fibrations}, all of them in the range of $h^{1,1} + h^{1,2} \le 502$. This can be regarded as striking evidence for an affirmative answer to the following question attributed to Yau: are there only finitely many families of irreducible Calabi-Yau $n$-folds? For related results in this directions we refer to \cite{allhodgenumbers, expandingzoo, quasismooth, quasismoothfiniteness, skarke}.

\medskip

The vast amount of examples of Calabi-Yau manifolds are so-called {\em CICY's}: (resolutions of) generic complete intersection Calabi-Yau varieties in Gorenstein toric Fano varieties. Here, reflexive polytopes play a key role and were introduced for this purpose by Batyrev \cite{Bat94}. To avoid confusion we adopt the notation in \cite{Bat94}. We will omit technical details, the interested reader is invited to look at the survey paper \cite{BN08}. 

\smallskip

An $s$-dimensional reflexive polytope $\Delta$ defines a Gorenstein toric Fano variety $X$ given by the fan over the faces of $\Delta^*$. A generic anticanonical hypersurface $Y$ in $X$ is a (possibly singular) Calabi-Yau variety of dimension $s-1$. For $s\le 4$ it can be (crepantly) resolved by a Calabi-Yau manifold $\widehat{Y}$. Exploiting the duality of reflexive polytopes, Batyrev showed that for $s=4$ the Calabi-Yau $3$-folds $\widehat{Y},\widehat{Y^*}$ constructed by $\Delta$ and $\Delta^*$ in this way have mirror-symmetric Hodge numbers: $h^{1,1}(\widehat{Y})=h^{1,2}(\widehat{Y^*})$. For $s > 4$, the possibly singular $Y$ may not be resolvable in this way. Therefore, one considers {\em stringy Hodge numbers} $\hstpq(Y)$, see \cite{dais, BB96}. In the case that $Y$ can be crepantly resolved by a Calabi-Yau manifold $\widehat{Y}$, the stringy Hodge numbers of $Y$ equal the usual Hodge numbers of $\widehat{Y}$.

\medskip

Batyrev and Borisov generalized Batyrev's results to complete intersections \cite{Bor93,BB96,BB97}. A generic CICY $Y$ of dimension $s-r$ in an $s$-dimensional Gorenstein toric Fano variety $X$ is given by a Minkowski decomposition of the reflexive polytope $\Delta = \Delta_1 + \cdots + \Delta_r$. Using such a datum, Batyrev and Borisov showed (in the notation of \cite{BN08}) that the {\em stringy $E$-polynomial} of $Y$ 
\begin{equation}
\label{poly}
\Est(Y) := \sum_{p,q} (-1)^{p+q} \, \hstpq(Y) \, u^p v^q
\end{equation}
equals $\Est(P)$, a rather complicated combinatorial expression called the {\em stringy $E$-poly\-no\-mi\-al} of an associated Gorenstein polytope $P$ of of index $r$ and dimension $d := s+r-1$. We refer to Section~\ref{CICY} for details of the Batyrev-Borisov construction.

\smallskip

It is important to remark that not every Gorenstein polytope is given by such a Minkowski decomposition, still, $\Est(P)$ is always of the form 
(\ref{poly}), so stringy Hodge numbers of Gorenstein polytopes are well-defined, see \cite{schepers}. Moreover, we define the {\em Calabi-Yau dimension} of {\em any} Gorenstein polytope $P$ of dimension $d$ and index $r$ as $n:=d+1-2r$. Gorenstein polytopes also satisfy a beautiful duality, and under additional hypotheses (the existence of a so-called {\em nef-partition}) it is possible to show that a CICY has an analogously constructed mirror partner (on the level of stringy Hodge numbers), for more 
on this see \cite{Bor93,oncicytoric,BB96,BB97,BN08}.

\smallskip

After these preparations, we can state the result which originally motivated our investigations. (We remark that $\Est(P) = 0$, if $n < 0$, see \cite{schepers}.)

\begin{theorem}
\label{main-hodge}
Let $P$ be a smooth Gorenstein polytope of Calabi-Yau dimension $n \ge 0$ and of dimension $d > 3 n+3$. 
Then $\Est(P)=\Est(P')$ for $P'$ a smooth Gorenstein polytope of Calabi-Yau dimension $n$ and of dimension at most $3n+1$. 
More precisely, $\Est(P)$ equals the $E$-polynomial of an $n$-dimensional Calabi-Yau manifold $Y$ given as the complete intersection in projective space $\P^{\tilde{s}}$ of generic hypersurfaces of degrees $d_1, \ldots, d_{\tilde{r}} \in \Z_{\ge 2}$ with $d_1+\ldots+d_{\tilde{r}}=\tilde{s}+1$, where $\tilde{s} \le 2n+1$. Any $n$-dimensional generic Calabi-Yau complete intersection associated to the Gorenstein polytope $P$ (in the sense of Batyrev-Borisov) is isomorphic to $Y$.
\end{theorem}

Since there are only finitely many Gorenstein polytopes in fixed dimension, this answers affirmatively for the special class of {\em smooth} Gorenstein polytopes Question 4.21 in \cite{BN08} asking whether there should be (up to multiples) only finitely many stringy $E$-polynomials of Gorenstein polytopes of given Calabi-Yau dimension. In particular, the validity of this conjecture for any Gorenstein polytope would imply the finiteness of stringy Hodge numbers of all irreducible CICY's given by the Batyrev-Borisov construction. In our present situation the previous theorem shows even more:

\begin{cor}
\label{cor-hodge}
There are only finitely many families of (possibly singular) generic Calabi-Yau complete intersections of dimension $n$ that are associated to {\em smooth} Gorenstein polytopes.
\end{cor}

Table~\ref{fano-table} gives a complete list of all smooth Gorenstein polytopes of Calabi-Yau dimension $n \le 3$ up to large $d$. Note that by Theorem~\ref{main-hodge} it is enough to consider $d \le 12$ for $n=3$. Of course, the assumption of smoothness is very strong and yields only very few Gorenstein polytopes. The subtle issue which of the stringy Hodge numbers of these Gorenstein polytopes are realized by CICY's is addressed in Section~\ref{realization}. In particular, we find $7$ Hodge numbers of Calabi-Yau $3$-folds that are not yet contained in the database \cite{jurke} by Benjamin Jurke.

\begin{remark}Let us remark that in \cite{skarke} the Kreuzer-Skarke algorithm was extended to a potential classification procedure of Gorenstein polytopes of given $d$ and $n$, however, it is unclear yet how computationally feasible this will be. Lists of Gorenstein polytopes with $n=3$ coming from so-called basic IP weight systems can be found on the webpage~\cite{skarke-database}.
\end{remark}

\bigskip

The paper is organized as follows: Section 2 gives a combinatorial proof of the classification of smooth Gorenstein polytopes with large index (Theorem~\ref{main-theo}). Section 3 deduces Proposition~\ref{main-fano} and Theorem~\ref{main-hodge}. Section 4 explains the classification algorithm. Finally, Section 5 discusses the realization of stringy Hodge numbers by CICY's. 

\medskip

\textbf{Acknowledgments:} The second author is supported by the US National Science Foundation (DMS 1203162) and would also like to thank Christian Haase and the research group Lattice Polytopes at the Goethe-Universit\"at Frankfurt for hospitality and fruitful discussions. We thank Cinzia Casagrande for references and for putting our results into their proper perspective. 

\section{Smooth Gorenstein polytopes of large index}

\subsection{Cayley polytopes}

\label{sec-cay}
For convenience, we will denote the lattice by $\Z^d$, as long as it is not important to differentiate between $M$ or $N$.\smallskip

For $d$ a positive integer we define the {\em unimodular $d$-simplex} as 
\[S_d := \conv(0,e_1,\ldots,e_d),\]
where $e_1, \ldots, e_d$ is a lattice basis.

\begin{definition}
\label{def-cay}
Let $P_0, \ldots, P_k \subset \R^s$ be lattice polytopes. 
Then we define the {\em Cayley polytope} of $P_0, \ldots, P_k$ as
\[P_0 * \cdots * P_k := \conv(P_0 \times e_0, \ldots, P_k \times e_k) \subseteq \R^s \oplus \R^{k+1},\]
where $e_0, \ldots, e_k$ is a lattice basis of $\R^{k+1}$. Here, $P_0, \ldots, P_k$ are called {\em Cayley factors}. Note that if the dimension of the affine span $\aff(P_0, \ldots, P_k)$ equals $s$, 
then the Cayley polytope is a lattice polytope of dimension $s+k$.
\end{definition}

Now, we can formulate the following result.

\begin{theorem}
\label{main-theo}
Let $P$ be a smooth Gorenstein polytope of dimension $d$ and index $r$. Then $r > \frac{d+3}{3}$ if and only if 
\begin{enumerate}
\item $P \cong S_d$ (here, $r=d+1$)
\item $P \cong 2 S_d$ with $d \ge 5$ odd (here, $r=\frac{d+1}{2}$)
\item $P \cong S_{\frac{d}{2}} \times S_{\frac{d}{2}}$ with $d \ge 4$ even (here, $r=\frac{d+2}{2}$) 
\item $P \cong (a_1+1) S_{d+1-r} * \cdots * (a_t+1) S_{d+1-r} * S_{d+1-r} * \cdots * S_{d+1-r},$
where there are $\frac{d+3}{3} < r \le \frac{d+1}{2}$ Cayley factors, and $a_1, \ldots, a_t$ is an integer partition of $d+2-2r$.
\end{enumerate}
Different integer partitions in (4) yield non-isomorphic Gorenstein polytopes.
\end{theorem}

\label{proofsec}

While strictly speaking not necessary, we will often express bounds in terms of the Calabi-Yau dimension $n:=d+1-2r$, since this is closer to the point of view of our paper. Note that 
\[r > \frac{d+3}{3} \quad\lolra\quad d > 3n+3\,.\]

We will give a combinatorial proof of Theorem~\ref{main-theo} in the remainder of this section.

\begin{remark}{\rm For the algebro-geometric reader this is how we will proceed. It follows from Mukai's conjecture \cite{casagrande,mukai-conj} that large index implies Picard number $\rho \le 2$. In the toric situation Kleinschmidt \cite{kleinschmidt} showed that these Fano manifolds have to be projective toric bundles. It remains to use the assumption on the index to specify them and their associated reflexive polytopes. This can be done using numerical criteria. The proof presented here follows along these lines. It is slightly longer, but more elementary and combinatorial.}
\end{remark}

\subsection{Bounding the number of vertices}

We denote by $\V$ the vertex set of a polytope. Let us consider a $d$-dimensional {\em simplicial} reflexive polytope $Q \subset \NR$ (i.e., 
dual to a simple reflexive polytope). Casagrande defined in \cite{casagrande} the number
\[\delta_Q := \min \{\pro{v}{u} \;:\; v \in \V(Q),\, u \in \V(Q^*), v \notin F_u\} \in \Z_{\ge 0},\]
where $F_u$ defines the facet of the dual polytope $Q^*$ corresponding to the vertex $u$ of $Q$. Here $\pro{\cdot}{\cdot}$ denotes the inner pairing of $N$ and $M$. 

Casagrande showed the following result (Theorem 3({\rm ii}) in \cite{casagrande}). For readers with a background in algebraic geometry, this proved the validity of Mukai's conjecture $\rho_X(i_X-1) \le d$ (where $\rho_X$ is the Picard number) for $\Q$-factorial Gorenstein toric Fano varieties $X$.

\begin{theorem}[Casagrande '06]
\label{casa}
If $\delta_Q > 0$, then
\[|\V(Q)| \le d + \frac{d}{\delta_Q}\,.\]
\end{theorem}

The following observation is also well-known (Lemma 2 in \cite{casagrande}):

\begin{lemma}
\label{inequality} 
Let $P$ be a simple Gorenstein polytope of index $r$. Then $(rP)^*$ is a simplicial reflexive polytope with
\[\delta_{(r P)^*} \ge r-1\,.\]
\end{lemma}

From these results it is straightforward to deduce an upper bound on the number of vertices of a simple Gorenstein polytope of small Calabi-Yau dimension.

\begin{proposition}
\label{main-prop}
Let $P$ be a simple Gorenstein polytope of Calabi-Yau dimension $n$ and of dimension $d > 3 n+3$. 
Then $(rP)^*$ is a simplex or has $d+2$ vertices.
\end{proposition}

\begin{proof}

If $r=1$, then $n=d-1$, so $3n+3=3d > d$, a contradiction. Hence, $r > 1$. 
Let us consider the simplicial reflexive polytope $Q := (r P)^*$. Then Theorem~\ref{casa} and Lemma~\ref{inequality} imply
\[|\V(Q)| \le d + \frac{d}{r-1} = d + \frac{2d}{d-1-n}\,.\]
The statement follows now from
\[\frac{2d}{d-1-n} < 3 \;\;\lolra\;\; 3 n + 3 < d\,.\qedhere\]
\end{proof}

\begin{remark}
\label{triple}
Let us note that the bound is sharp: for any $n\ge 0$, the smooth Gorenstein polytope $P=S_{n+1} \times S_{n+1} \times S_{n+1}$ has dimension $d=3n+3$, Calabi-Yau dimension $n$ and $d+3$ 
vertices.
\end{remark}

\subsection{Proof of Theorem~\ref{main-theo}}

By Proposition~\ref{main-prop} we only have to consider two cases in order to prove Theorem~\ref{main-theo}: either $P$ is a simplex or it 
has $d+2$ facets.

\begin{proposition}
\label{prop-simplex}
Let $P$ be a smooth Gorenstein simplex of Calabi-Yau dimension $n$ and of dimension $d \ge 3n$. Then either $n<0$ and $P\cong S_d$, or $n=0$ and $P \cong 2 S_d$ (with odd $d$).
\end{proposition}

\begin{proof}
In this case, $X_P \cong \P^d$. As is well-known, $P \cong k S_d$ for some $k \in \Z_{\ge 1}$. 
Since, $r P \cong (d+1) S_d$, we get $rk=d+1$. Therefore, $\frac{d+1-n}{2} k = d+1$. For $n \le -1$, we get $k<2$, so $k=1$. 
For $n=0$, we have $k=2$, as desired. Finally, let $n > 0$, so $k \ge 3$. In particular, 
$\frac{k}{k-2} \le 3$. This implies
\[d+1  = \frac{k}{k-2} n \le 3n\,,\]
a contradiction.
\end{proof}

The following statement can be deduced from \cite[Theorem~2.6]{BN08}. Recall that $S_d$ is a smooth Gorenstein polytope of index $d+1$.

\begin{lemma}
\label{cayley-comb}
If $b_1, \ldots, b_r \in \Z_{>0}$, $\sum_{i=1}^r b_i = d+2-r$, and $r \in \{1, \ldots, d+1\}$, then 
\[b_1 S_{d+1-r} * \cdots * b_r S_{d+1-r}\,,\]
is a smooth $d$-dimensional Gorenstein polytope of index $r$ such that $(rP)^*$ has $d+2$ vertices.
\end{lemma}

In our situation, the converse also holds.

\begin{proposition}
\label{prop-circuit}
Let $P$ be a smooth Gorenstein polytope of dimension $d \ge 3n+3$ such that $(rP)^*$ has $d+2$ vertices. Then 
$n \ge -1$ and there exists a unique integer partition $a_1, \ldots, a_t$ of $n+1$ such that 
\[P \cong (a_1+1) S_{d+1-r} * \cdots * (a_t+1) S_{d+1-r} * S_{d+1-r} * \cdots * S_{d+1-r}\,,\]
where there are $r$ Cayley factors.
\end{proposition}

Recall that we consider integer partitions as multisets (so uniqueness is up to permutation).

\begin{proof}
By Kleinschmidt's classification of nonsingular toric Fano varieties whose associated fan has $d+2$ rays \cite{kleinschmidt} (see also Section 7.3 in \cite{toricbook}), we can assume that the vertices of the reflexive polytope $(rP)^*$ are of the following form:
\[
e_1, \ldots, e_d, \qquad v_1 = - e_1 - \cdots - e_k, \qquad
v_2 = a_1 e_1 + \cdots + a_k e_k - e_{k+1} - \cdots - e_d,
\]
where $1 \le k \le d-1$, and $a_i \in \Z_{\ge 0}$ for $1 \le i \le k$ such that
$ m := \sum_{i=1}^k a_i \le d-k\,. $ 

\smallskip
Combinatorially, $(rP)^*$ is a free sum (i.e., dual to the product) of a $k$-dimensional simplex and a $(d-k)$-dimensional simplex. 
This implies that the vertices of $rP$ are of the following types, where all entries not specified are equal to $-1$ 
(here $i \in \{1, \ldots, k\}$, $j \in \{k+1, \ldots, d\}$):
\begin{center}\small
\newcolumntype{C}{>{\begin{math}}c<{\end{math}}}
\begin{tabular}{CCCCCCCl}
       & &i& &j& & &corresponding facet(s) of $(rP)^*$\\
       \midrule
   (-1,&\ldots&-1,&\ldots&-1,&\ldots&-1)&not containing $v_1$ and $v_2$\\
   (-1,&\ldots&-1,&\ldots&d-k-m,&\ldots&-1)&not containing $v_1$ and $e_j$\\
   (-1,&\ldots&k,&\ldots&-1,&\ldots&-1)&not containing $e_i$ and $v_2$\\
   (-1,&\ldots&k,&\ldots&a_i(k+1)+d-k-m,&\ldots&-1)&not containing $e_i$ and $e_j$
\end{tabular}
\end{center}

Let us translate the first vertex $v$ of $rP$ into the origin. This yields the following vertices for the lattice polytope $(rP - v)/r$ (which is isomorphic to $P$):
\begin{center}
\newcolumntype{C}{>{\begin{math}}c<{\end{math}}}
\begin{tabular}{CCCCCCC}
       & &i& &j& & \\
       \midrule
   (0,&\ldots&0,&\ldots&0,&\ldots&0)\\
   (0,&\ldots&0,&\ldots&\frac{d+2-m-(k+1)}{r},&\ldots&0)\\
   (0,&\ldots&\frac{k+1}{r},&\ldots&0,&\ldots&0)\\
   (0,&\ldots&\frac{k+1}{r},&\ldots&\frac{(a_i -1)(k+1)+d+2-m}{r},&\ldots&0)
\end{tabular}
\end{center}

Note that $r$ divides $k+1$, and thus $d+2-m$. By our assumption 
\[ 3n+3 \le d=n+2r-1 \quad \Longrightarrow \quad n+1 \le r-1,\]
so $d+2-m \le d+2 =n+2r+1 \le 3r-1$. Hence, $d+2-m$ equals $r$ or $2r$. 
Furthermore, by our assumption $m\le d-k$, so $d-m+2 > k+1$, thus,
\[0 < \frac{d+2-m}{r} - \frac{k+1}{r} \le 1,\]
which implies $d+2-m=2r$ and $k+1=r$. Therefore, $P$ is isomorphic to the Cayley polytope 
\[S_{d+1-r} * (a_1+1) S_{d+1-r} * \cdots * (a_k+1) S_{d+1-r},\]
where $\sum_{i=1}^k a_i = m=n+1$. Moreover, by construction, the polytopes do not depend on the ordering of these numbers, so by 
choosing $t \le k$ non-zero $a_i$'s the result follows. Finally, different multisets of coefficients define non-isomorphic Cayley polytopes, 
since
\[(-a_1+1)e_1 + \cdots + (-a_k+1)e_k + e_{k+1} + \cdots + e_d + v_1 + v_2 = 0\]
is the unique affine relation of the vertices of $(rP)^*$, see also \cite{kleinschmidt}. 
\end{proof}

\begin{proof}[Proof of Theorem~\ref{main-theo}]

If $r > \frac{d+3}{3}$, then Proposition~\ref{prop-simplex} implies cases (1) and (2), while Proposition~\ref{prop-circuit} yields (3) for $n=-1$ (all $a_i$'s are equal to $0$), and (4) for $n \ge 0$. The `if'-statement follows for (4) from Lemma~\ref{cayley-comb}.\end{proof}


It is a somewhat lucky coincidence that the essentially same bound appears quite naturally in Propositions~\ref{main-prop} and \ref{prop-circuit}. One should also compare this with the fact \cite[Lemma~2]{skarke} that any so-called basic IP weight system (the building blocks of Gorenstein polytopes in Skarke's classification algorithm) satisfy $d \le 3n-1$.

\section{Applications to Fano manfolds and Calabi-Yau varieties}

\subsection{Toric Fano manifolds of large index}

Let us prove Proposition~\ref{main-fano} by translating Theorem~\ref{main-theo} into algebraic geometry. 

\begin{lemma}
\label{cayley-alggeo}
Let $s,l \in \Z_{>0}$, and $c_1, \ldots, c_l \in \N$. Then 
\[S_s * (c_1+1) S_s * \cdots *  (c_l+1) S_s \subset \MR\]
defines a toric projective bundle
\[\P_{\P^s}(O \oplus O(c_1) \cdots \oplus O(c_l))\,.\]
\end{lemma}

\label{fanosec}

For a proof of this fact see \cite{kleinschmidt} or \cite[Section 7.3]{toricbook}.

\begin{proof}[Proof of Proposition~\ref{main-fano}]

Let $X$ be a toric Fano $d$-fold with $r := i_X > \frac{d+3}{3}$. Then there exists a smooth Gorenstein polytope $P \subset \MR$ such that 
the reflexive polytope $rP$ is associated to the anticanonical divisor $-K_X$. Now, we apply Theorem~\ref{main-theo} and the previous Lemma. The converse statement follows 
from combining the previous Lemma with Lemma~\ref{cayley-comb}.
\end{proof}

Let us note that we recover a result of Wi\'sniewski \cite{fanolargeindex} in this toric setting. 

\begin{cor}
\label{fano-large}
Let $X$ be a toric Fano $d$-fold with $i_X \ge \frac{d+1}{2}$. Then $X$ is isomorphic to one of the following cases:
\begin{itemize}
\item $\P^1 \times \P^1 \times \P^1$, $d=3$, and $i_X=2$
\item $\P_{\P^{\frac{d+1}{2}}}(O(1) \oplus O^{\frac{d-1}{2}})$, $d$ is odd, and $i_X=\frac{d+1}{2}$
\item $\P^{\frac{d}{2}} \times \P^{\frac{d}{2}}$, $d \ge 4$ is even, and $i_X=\frac{d+2}{2}$
\item $\P^d$, and $i_X=d+1$
\end{itemize}
\end{cor}

\begin{proof}
The case $d=3$ is well-known, it can also be easily checked using the database. 
Since $d > 3$ is equivalent to $\frac{d+1}{2} > \frac{d+3}{3}$, the statements follow immediately from Proposition~\ref{main-fano}.
\end{proof}

\subsection{Calabi-Yau complete intersections}

\label{CICY}

Let us recall the so-called Ba\-ty\-rev-Borisov construction (see \cite{BN08} for a survey). Let $\Delta \subset \MR$ be a reflexive polytope of dimension $s$ and lattice polytopes $\Delta_1, \ldots, \Delta_r \subset \MR$ such that 
\[\Delta = \Delta_1 + \cdots + \Delta_r,\]
where we add sets pointwise (Minkowski summation). Now, the intersection of the generic hypersurfaces given by the generic Laurent polynomials 
\[F_{\Delta_i}(z) := \sum_{m \in \Delta_i \cap M} c_m X^m \quad\quad (\text{for } c_m \in \C^*),\]
for $i=1,\ldots, r$ defines a complete intersection in $(\C^*)^s$, whose compactification in $\P^s$ is a Calabi-Yau variety $Y$ (a CICY) of dimension $n := s-r$. On the other hand, one can define the Cayley polytope $P := \Delta_1 * \cdots * \Delta_r$, which is a Gorenstein polytope of dimension $s+r-1$ and of index $r$. Note that $P$ has 
Calabi-Yau dimension $d+1-2r=n$. Now, the stringy $E$-polynomial of $Y$ equals the stringy $E$-polynomial of $P$ (as defined in \cite{BN08} based on \cite{stringcohomology}, see also~\cite{schepers}). 

\medskip

The proof of Theorem~\ref{main-hodge} will be a direct consequence of the following observation. It is straightforward from an algebro-geometric viewpoint, however 
it seems to be quite a challenge to prove it using combinatorics only!

\begin{lemma}
\label{reduction}
Let $b_2 + \cdots + b_r = s$ (all these numbers being positive integers). Then 
\[S_s * b_2 S_s * \cdots * b_r S_s \quad\text{ and }\quad b_2 S_{s-1} * \cdots * b_r S_{s-1}\]
are smooth Gorenstein polytopes with the same Calabi-Yau dimension and the same stringy $E$-polynomial. More precisely, the CICY's given by 
$S_s + b_2 S_s + \cdots + b_r S_s$, respectively by $b_2 S_{s-1} + \cdots + b_r S_{s-1}$, are isomorphic.
\end{lemma}

\begin{proof}
In this case, the Minkowski sum $S_s + b_2 S_s + \cdots + b_r S_s$ defines a generic Calabi-Yau complete intersection $Y$ in $\P^s$ 
given by generic hypersurfaces of degrees $1,b_2,\ldots,b_s$. Identifying the generic hypersurface of degree $1$ with $P^{s-1}$, we note that $Y$ can be regarded as a generic complete intersection of hypersurfaces of degrees $b_2, \ldots, b_s$ in $\P^{s-1}$. This finishes the proof.
\end{proof}

\begin{proof}[Proof of Theorem~\ref{main-hodge}]

Note that $n \ge 0$ if and only if $r \le \frac{d+1}{2}$, so only cases (2) and (4) may occur in Theorem~\ref{main-theo}. In case (2), we have $n=0$, and $\Est(2 S_d)=2=\Est([-1,1])$, see e.g.~Example~4.12 in \cite{BN08}. Hence, it suffices to consider case (4). We may assume that $a_1, \ldots, a_t$ are all positive integers (since their sum equals $n+1 \ge 1$). Now, Lemma~\ref{reduction} shows that it suffices to consider $t=r$ (note again that $a_i=0$ for $i >t$). Therefore,  $n+1=\sum_{i=1}^r a_i \ge r$, thus $d = (d+1-2r)+2r-1=n+2r-1 \le n+2(n+1)-1=3n+1$. This proves the first statement. The second claim follows from $s = d+1-r =n+r \le 2n+1$ and the proof of Lemma~\ref{reduction}. 

For the last statement of the theorem, let $P$ be a Gorenstein polytope associated to a Minkowski decomposition $\Delta = \Delta_1 + \cdots + \Delta_r$ as above, i.e.,
\begin{equation}P = \Delta_1 * \cdots * \Delta_r\,. \label{cayley1}\end{equation}
Let $d > 3n+3$. Again, we have two cases. 
If $n=0$, then $P \cong 2 S_d$, hence $r=1$ (since $2 S_d$ is not a Cayley polytope). Therefore, $d=2$, a contradiction 
to $d > 3n+3$. Therefore, we are in case (4) of Theorem~\ref{main-theo}, so by Definition~\ref{def-cay}
\begin{equation}P \cong \conv((a_1+1) S_s \times e_1, \ldots, (a_t+1) S_s \times e_t, S_s \times e_{t+1}, \ldots, S_s \times e_r),\label{cayley2}\end{equation}
where  $a_1, \ldots, a_t$ is an integer partition of $n+1$. Equation (\ref{cayley1}) implies the 
existence of a surjective, affine lattice homomorphism $\phi \;:\; P \to S_{r-1}$.  We may assume $a_1 > 0$. 
Since for every vertex of $(a_1+1)S_s \times e_1$ every adjacent edge contains a lattice point in its relative interior, $\phi((a_1+1)S_s \times e_1)$ has to be a vertex of $S_{r-1}$. Hence,  $S_s \times e_1$ lies in the fiber space of $\phi$. Therefore, every factor in expression (\ref{cayley2}) maps onto a vertex via $\phi$. Since 
the map is surjective and there are $r$ factors, it follows that these two Cayley decompositions are the same (up to a permutation of the factors). Hence, the Minkowski decomposition of $\Delta$ equals up to permutation and translations by lattice points the Minkowski decomposition $(s+1)S_s = (a_1+1) S_s + \cdots + (a_t+1) S_s + S_s + \cdots + S_s$. This implies the statement (e.g., as in the proof of Lemma~\ref{reduction}).
\end{proof}

In the notation of \cite{BB97,BN08,skarke}, this shows that for $d > 3n+3$ the `reflexive Gorenstein cone' over $P$ is `completely split'.

\section{The classification algorithm for smooth Gorenstein polytopes of given index}

Let us briefly describe the modified version of the algorithm to classify smooth Fano polytopes by Mikkel \O{}bro~\cite{oebro}.
The key ingredient for his algorithm is the notion of a special facet.
\begin{definition}
Let $P$ be a reflexive polytope, a facet $F$ of $P$ is called \emph{special}, if the sum of all vertices $v_1,\ldots,v_k$ of $P$ is a point within the cone over the facet $F$, i.e., $\sum_{i=0}^k v_i \in \cone(F)\,$.
\end{definition}
Clearly, every reflexive polytope has at least one special facet.
Now, let $P$ be a smooth reflexive $d$-polytope. Then the dual polytope $P^*$ is a simplicial reflexive polytope where the vertices of each facet form a lattice basis of $\Z^d$. Let us denote here such reflexive polytopes as {\em dual-smooth}. 
By applying a unimodular transformation, we can assume that $\conv(e_1,\ldots,e_d)$ is a special facet of $P^*$.
\O{}bro has shown that all remaining vertices of such a dual-smooth reflexive polytope must lie within some explicit finite set.
The original algorithm enumerates this set in a clever way to efficiently generate all dual-smooth reflexive polytopes.
Now, because minus the all-one vector defines the special facet we used for the canonicalization (since it evaluates to $-1$), by Lemma~\ref{inequality} we know that every other vertex $v$ of $P^*$ must evaluate to $\pro{-\onev}{v} \ge r-1$.
Using this property we can further restrict the set of possible vertices. 
By adding this condition to the original algorithm we were able generate a superset of all simplicial reflexive polytopes whose dual polytopes are $r$-multiples of smooth Gorenstein polytopes for given dimension $d$ and index $r$ larger than some $r_0$. For $d=10$ and $r\ge3$ this algorithm 
took about 1 hour. This should be compared to two weeks for the original (parallelized) algorithm without additional conditions for $d=9$. 
Table~\ref{fano-table} contains all isomorphism classes for values of $d$ and $r$ that we could compute so far.

\section{Realization of stringy Hodge numbers by Calabi-Yau complete intersections}

Table~\ref{fano-table} combined with Theorem~\ref{main-theo} gives a complete enumeration of all smooth Gorenstein polytopes $P$ with 
Calabi-Yau dimension $0 \le n \le 3$. In this section, we would like to discuss which stringy Hodge numbers of these smooth Gorenstein polytopes 
are possibly realized by Calabi-Yau complete intersections.

\smallskip

First, let us recall that every Gorenstein polytope $P$ has a {\em dual Gorenstein polytope} $P^\times$ of same dimension, same index, and same Calabi-Yau dimension $n$, see \cite{BB97} or \cite{BN08}, satisfying the combinatorial 
mirror symmetry 
\[\Est(P;u,v) = (-u)^n \Est(P^\times;u^{-1},v)\,.\]
In particular, for $n=3$, the stringy Hodge numbers $(h^{1,1},h^{1,2})$ are interchanged between $P$ and $P^\times$, see \cite{BB96, BN08, schepers}. Figure~\ref{hodge-P-dual-table} shows the complete list of pairs $(h^{1,1},h^{1,2})$ of stringy Hodge numbers of $P^\times$ for $n=3$ and $d \le 12$. Note that by Theorem~\ref{main-hodge} this is the complete list for $n=3$.

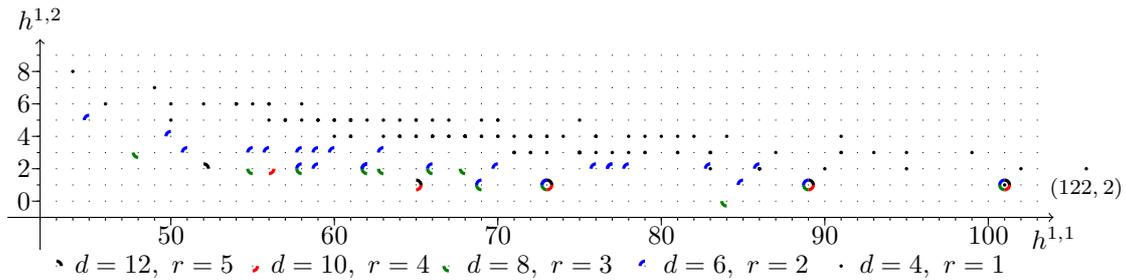
\begin{figure}[tbp]
\begin{center}
{ \footnotesize
\begin{tikzpicture}[scale=0.215]

\draw[->] (40,-1) -- (104,-1) node[anchor=north] {$h^{1,1}$};
\draw[->] (42,-3) -- (42,10) node[anchor=south] {$h^{1,2}$};

\foreach \i in {42,43,...,103} {
	\draw (\i,-1.2) -- (\i,-1);
}
\foreach \i in {0,1,2,...,9} {
	\draw (41.8,\i) -- (42,\i);
}

\foreach \i in {50,60,...,100} {
	\draw (\i,-1.5) -- (\i,-1) node[anchor=north] {$\i$};
}
\foreach \i in {0,2,4,...,8} {
	\draw (41.5,\i) -- (42,\i) node[anchor=east] {$\i$};
}

\foreach \i in {43,44,...,103} {
	\foreach \j in {0,1,2,...,9} {
		\fill (\i,\j) circle (1pt);
	}
}

\foreach \pair in 
{(99,3),(67,4),(87,3),(91,4),(67,5),(81,4),(79,4),(95,3),(93,3),(69,4),(82,4),(73,4),(91,3),(73,4),(70,4),(61,4),(76,3),(83,3),(71,3),(95,2),(57,5),(66,4),(68,4),(57,5),(46,6),(75,3),(61,4),(84,4),(65,5),(76,4),(72,4),(91,3),(72,4),(59,5),(58,5),(61,4),(50,5),(60,4),(66,4),(68,4),(73,3),(72,3),(79,3),(64,4),(75,3),(86,2),(75,5),(58,6),(69,5),(80,4),(78,4),(61,5),(70,5),(56,6),(70,5),(56,6),(66,5),(62,5),(76,4),(60,5),(52,6),(64,5),(60,5),(65,5),(54,6),(50,6),(63,5),(57,5),(71,4),(68,4),(72,4),(64,4),(81,3),(55,6),(49,7),(44,8),(60,5),(54,6),(61,5),(65,4),(74,4),(68,4),(56,5),(64,4),(72,3),(65,4),(69,4),(68,4),(68,4),(70,4),(59,5),(63,4),(78,3),(76,4),(61,5),(72,4),(61,5),(54,6),(59,5),(66,4),(67,4),(66,4),(75,3),(83,3),(83,3),(79,3),(81,3),(71,3),(68,4),(64,4),(75,3),(75,3),(71,3),(86,2),(102,2),(77,3),(86,2),(68,4),(75,3),(90,2),(86,2),(83,2),(101,1)}
{
	\fill[black] \pair circle (3pt);
}

\fill[black] (106,2) node[anchor=north] {{\tiny$(122,2)$}} circle (3pt) ;

\foreach \pair in 
{(101,1),(86,2),(78,2),(77,2),(83,2),(77,2),(63,3),(59,2),(85,1),(76,2),(70,2),(58,2),(59,3),(62,2),(73,1),(60,3),(56,3),(58,3),(50,4),(51,3),(62,2),(62,2),(45,5),(55,3),(66,2),(89,1),(73,1),(69,1)}
{
	\draw[blue,very thick] \pair +(0.0,0.3) arc [radius=0.3,start angle=90,end angle=180];
	
}


\foreach \pair in 
{(101,1),(89,1),(73,1),(68,2),(58,2),(62,2),(66,2),(55,2),(63,2),(69,1),(48,3),(84,0)}
{
	\draw[green!50!black,very thick] \pair +(-0.3,0.0) arc [radius=0.3,start angle=180,end angle=270];
}


\foreach \pair in 
{(101,1),(89,1),(73,1),(65,1),(56,2)}
{
	\draw[red,very thick] \pair +(0.0,-0.3) arc [radius=0.3,start angle=270,end angle=360];
}


\foreach \pair in
{(101,1),(89,1),(73,1),(65,1),(52,2)}
{
	\draw[black,very thick] \pair +(0.3,0.0) arc [radius=0.3,start angle=0,end angle=90];
}

\fill (43,-4) circle (1pt) node[anchor=west] {$\ d = 12,\ r=5$};
\fill (55,-4) circle (1pt) node[anchor=west] {$\ d = 10,\ r=4$};
\fill (67,-4) circle (1pt) node[anchor=west] {$\ d = 8,\ r=3$};
\fill (79,-4) circle (1pt) node[anchor=west] {$\ d = 6,\ r=2$};
\fill[fill=black] (91,-4) circle (3pt) node[anchor=west] {$\ d = 4,\ r=1$};
\draw[blue,very thick] (79,-4) +(0.0,0.3) arc [radius=0.3,start angle=90,end angle=180];
\draw[green!50!black,very thick] (67,-4) +(-0.3,0.0) arc [radius=0.3,start angle=180,end angle=270];
\draw[red,very thick] (55,-4) +(0.0,-0.3) arc [radius=0.3,start angle=270,end angle=360];
\draw[black,very thick] (43,-4) +(0.3,0.0) arc [radius=0.3,start angle=0,end angle=90];
\end{tikzpicture}}\vspace{-2mm}
\caption{Stringy Hodge numbers of the duals of smooth Gorenstein polytopes with Calabi-Yau dimension 3.\label{hodge-P-dual-table}}
\end{center}
\end{figure}

\medskip
To compute the stringy Hodge numbers of these polytopes we implemented the formula for the stringy $E$-polynomial described in~\cite[Definition 4.8]{BN08} in the \polymake framework~\cite{polymake}, using an interface to Normaliz~\cite{normaliz2,normaliz2paper} for the Ehrhart $h^*$-polynomials.

\begin{remark}{\rm 
We note that for $n=3$ the $12$-dimensional smooth Gorenstein polytope $P$ from Remark~\ref{triple} has Hodge-pair $(2,52)$ which did not appear for any of the polytopes of lower dimension, indicating that 
the bound $3n+3$ in Theorem~\ref{main-hodge} could indeed be sharp. However, for general $n$ we cannot rule out yet that the stringy $E$-polynomial of this example might appear for some smooth Gorenstein polytope of lower dimension. In fact, for $n=3$ there exists a $6$-dimensional (non-smooth) Gorenstein polytope with this pair of Hodge numbers~\cite{skarke-database}.
}
\end{remark}


As it turns out, all of these `virtual' stringy Hodge numbers indeed equal the Hodge numbers of Calabi-Yau manifolds. 
To prove this, let us recall the definition of normality.

\begin{definition}A $d$-dimensional lattice polytope $P \subseteq \R^d$ is {\em normal} (or {\em integrally closed}), 
if the semigroup of lattice points in the cone $C_P := \R_{\ge0}(P \times \{1\}) \subset \R^{d+1}$ is generated by lattice points in $C_P \cap \R^d \times \{1\}$.\label{def-normal}
\end{definition}

\begin{proposition}
\label{realizable}
Let $P$ be a Gorenstein polytope of Calabi-Yau dimension $n$. If $P$ is normal, then there exists a Calabi-Yau variety $Y^*$ of dimension $n$ such that $\Est(P^\times)$ equals the stringy $E$-polynomial of $Y^*$. Moreover, if $n\le 3$, then we can assume that $Y^*$ is smooth.
\end{proposition}

\begin{proof}

As follows from \cite[Cor.2.12]{BN08}, $P^\times$ is a Cayley polytope of length $r$, say, $P^\times = \nabla_1 * \cdots * \nabla_r$ in the notation of \cite{BN08}. 
Therefore, e.g.~by \cite[Thm.2.6]{BN08}, $\nabla_1 + \cdots + \nabla_r =: Q \subset \NR$ is a reflexive polytope of dimension $s := d+1-r$. In particular, there exists an associated generic $n$-dimensional CICY $Y^*$ in the Gorenstein Fano toric variety $X^*$ associated to $Q$. By the very definition of the stringy $E$-polynomial of Gorenstein polytopes, $\Est(Y^*)=\Est(P^\times)$. 

\smallskip

Let $n \le 3$. Let us recall the argument given in \cite{Bat94}. We choose a maximal projective crepant partial desingularization (MPCP) $\widehat{X^*}$ of $X^*$ in the sense of \cite{Bat94}. This induces an MPCP from a CICY $\widehat{Y^*}$ in $\widehat{X^*}$ to $Y^*$, where $\widehat{Y^*}$ is again a Calabi-Yau variety whose stringy $E$-polynomial equals $\Est(Y^*)$. Note that the cones of dimension $\le 3$ of the fan corresponding to $\widehat{X^*}$ are unimodular, so the toric strata of $\widehat{X^*}$ of dimension $\ge s-3$ are smooth 
(e.g., \cite[Thm.2.2.9]{Bat94}). Since $\widehat{Y^*}$ is generic of dimension $\le 3$, it avoids toric strata of dimension $< s-3$, hence $\widehat{Y^*}$ is smooth.
\end{proof}

There is a famous open conjecture: {\em Smooth polytopes are normal}. We verified it for all smooth Gorenstein polytopes we could compute. In particular, by Theorem~\ref{main-hodge} this observation implies:

\label{realization}
\begin{cor}
All stringy $E$-polynomials of duals of smooth Gorenstein polytopes with Calabi-Yau dimension $n\le 3$ equal the $E$-polynomial of some Calabi-Yau manifold of dimension $n$ in a Gorenstein toric Fano variety of dimension $s \le 2n+2$.
\end{cor}

\begin{proof}
It remains to consider the case $s \ge 2n+3$, hence $r=s-n \ge n+3$, thus $d=s+r-1 \ge 3n+5$, in which case Theorem~\ref{main-hodge} 
implies the statement.
\end{proof}

\begin{remark}
\label{hodge-polytopes} Among the Hodge numbers listed in Figure~\ref{hodge-P-dual-table} the pairs 
\[(84,0), (85,1), (52,2), (69,1), (65,1), (55,2), (63,2)\]
are not yet contained in the online database `Calabi-Yau $3$-fold explorer' by Benjamin Jurke \cite{jurke}. The associated smooth Gorenstein polytopes $P$ can be found on the webpage \cite{sgpdb}. We note that for the first five of these pairs Gorenstein polytopes with these stringy Hodge numbers can also be found in \cite{skarke-database}.
\end{remark}

\begin{remark} Among the stringy Hodge-numbers of the smooth Gorenstein polytopes themselves (i.e., not their duals), 
only the pairs $(1,69)$ and $(1,85)$ are not contained in the database of known Hodge numbers of Calabi-Yau $3$-folds \cite{jurke}. However, all three corresponding dual Gorenstein polytopes $P^\times$ are neither normal, nor do they contain a special $(r-1)$\nobreakdash-simplex (in the sense of \cite{BN08}). Hence, we cannot deduce whether there are CICY's associated to these smooth Gorenstein polytopes $P$.
\end{remark}

\begin{remark}{\rm The complete intersection Calabi-Yau manifolds described in Theorem~\ref{main-hodge} are given by a so-called {\em nef-partition}, see \cite{Bor93, BN08}. In particular, there exist Calabi-Yau manifolds with mirror-symmetric Hodge numbers. This does not have to be the 
case in general. For instance, there is a generic CICY with Hodge numbers $(84,0)$ (see Remark~\ref{hodge-polytopes}), 
however, there is no CICY with Hodge numbers $(0,84)$. We refer to \cite{BB97} for more on this rigidity-phenomenon.
}
\end{remark}

\def\bibfont{\small}
\bibliography{smooth-gorenstein-polytopes-arxiv-v1}{}
\bibliographystyle{amsalpha}
\end{document}